\documentclass{amsart}
\usepackage{amsmath,amsfonts,amsthm,amssymb,tabularx}
\usepackage[most]{tcolorbox}
\usepackage[utf8x]{inputenc}
\usepackage{ucs}
\usepackage{color}
\usepackage{url}
\definecolor{keywordcolor}{rgb}{0.7, 0.1, 0.1}   
\definecolor{commentcolor}{rgb}{0.4, 0.4, 0.4}   
\definecolor{symbolcolor}{rgb}{0.0, 0.1, 0.6}    
\definecolor{sortcolor}{rgb}{0.1, 0.5, 0.1}      
\usepackage{caption}

\usepackage{listings}

\newcommand{\bbn}{\mathbb{N}}
\newcommand{\bbz}{\mathbb{Z}}
\newcommand{\abs}[1]{\left\lvert #1\right\rvert}

\newcommand{\brac}[1]{\left( #1\right)}
\newtheorem{theorem}{Theorem}
\newtheorem{lemma}{Lemma}
\newtheorem{corollary}{Corollary}
\newtheorem{proposition}{Proposition}

\newcommand{\mathlib}{\textsf{mathlib}}

\makeatletter
\let\@wraptoccontribs\wraptoccontribs
\makeatother

\begin{document}

\title{On a density conjecture about unit fractions}
\author{Thomas F. Bloom}
\email{bloom@maths.ox.ac.uk}
\address{Mathematical Institute\\Woodstock Road\\Oxford\\OX2 6GG, United Kingdom}
\contrib[with an appendix co-written by]{Bhavik Mehta}

\begin{abstract}
We prove that any set $A\subset \bbn$ of positive upper density contains a finite $S\subset A$ such that $\sum_{n\in S}\frac{1}{n}=1$, answering a question of Erd\H{o}s and Graham.
\end{abstract}

\maketitle

\section{Introduction}

A classical topic in the additive theory of unit fractions is the study of those finite sets $S\subset \bbn$ such that $\sum_{n\in S}\frac{1}{n}=1$. Erd\H{o}s and Graham \cite{ErGr1980} conjectured that if the integers are coloured by a finite number of colours then there must be such a set which is monochromatic (aside from the trivial $S=\{1\}$). This conjecture was proved by Croot \cite{Cr2003}. 
\begin{theorem}[Croot \cite{Cr2003}]\label{th-croot}
If the integers $\{2,3,\ldots\}$ are coloured in finitely many colours there is a monochromatic set $S$ such that $\sum_{n\in S}\frac{1}{n}=1$.
\end{theorem}
Colouring statements often have a natural strengthening to a corresponding density statement, which Erd\H{o}s and Graham also conjectured\footnote{This conjecture appears on \cite[p. 36]{ErGr1980} and \cite[p. 298]{Gr2013}, for example, as well as in several unpublished notes and letters of Erd\H{o}s as reproduced in \cite{Gr2013}. This conjecture has also been made by Sun \cite{Su2007}.} in this case: that any set of positive density contains some finite $S$ with $\sum_{n\in S}\frac{1}{n}=1$. Recall that, for any $A\subset \bbn$, the lower and upper densities of $A$ are defined to be
\[\underline{d}(A) = \liminf_{N\to\infty}\frac{\abs{A\cap [1,N]}}{N}\quad\textrm{and}\quad\overline{d}(A) = \limsup_{N\to\infty}\frac{\abs{A\cap [1,N]}}{N}\]
respectively. When these limits are equal, $\underline{d}(A)=\overline{d}(A)=d$, we say that $A$ has density $d$ (and say that $A$ has positive density if $d>0$).

In this paper we answer this question of Erd\H{o}s and Graham in the affirmative in a strong sense, proving the following criterion which uses only positive upper density.
\begin{theorem}\label{th-dens}
If $A\subset \bbn$ has positive upper density (and hence certainly if $A$ has positive density) then there is a finite $S\subset A$ such that $\sum_{n\in S}\frac{1}{n}=1$.
\end{theorem}

Given that we are working with sums of the shape $\sum\frac{1}{n}$, it is perhaps more natural to consider the logarithmic density. For any $A\subset \bbn$, the lower and upper logarithmic densities of $A$ are defined to be
\[\underline{\delta}(A) = \liminf_{N\to\infty}\frac{1}{\log N}\sum_{n\in A\cap [1,N]}\frac{1}{n}\quad\textrm{and}\quad\overline{\delta}(A) = \limsup_{N\to\infty}\frac{1}{\log N}\sum_{n\in A\cap [1,N]}\frac{1}{n}.\]
A straightforward application of partial summation shows that $\overline{\delta}(A)\leq \overline{d}(A)$, and hence the analogue of Theorem~\ref{th-dens} with upper logarithmic density is an immediate corollary. Our methods in fact deliver the following more precise quantitative result.
\begin{theorem}\label{th-quant}
There is a constant $C>0$ such that the following holds. If $A\subset \{1,\ldots,N\}$ and 
\[\sum_{n\in A}\frac{1}{n}\geq C \frac{\log \log\log N}{\log\log N}\log N\]
then there is an $S\subset A$ such that $\sum_{n\in S}\frac{1}{n}=1$.
\end{theorem}
Let $\lambda(N)$ be the maximum possible size of $\sum_{n\in A}\frac{1}{n}$ if $A\subset \{1,\ldots,N\}$ has no subset $S$ with $\sum_{n\in S}\frac{1}{n}=1$. The primes show that\footnote{We use Vinogradov's notation $f\ll g$ to denote the existence of some absolute constant $C>0$ such that $\abs{f(n)}\leq C\abs{g(n)}$ for all sufficiently large $n$.} $\lambda(N) \gg \log\log N$. Pomerance has observed\footnote{To our knowledge, this observation has only appeared in the literature in the PhD thesis of Croot \cite{Crootphd}, who heard it from Pomerance via personal communication. We include a proof of this lower bound in an appendix  for the convenience of the reader.}  that taking all $n$ with a prime factor $p\gg n/\log n$ provides the lower bound $\lambda(N) \gg (\log\log N)^2$. This is the best lower bound that we are aware of, so that our current state of knowledge is now
\[(\log\log N)^2 \ll \lambda(N) \ll \frac{\log \log\log N}{\log\log N}\log N.\]
 We suspect that the lower bound is closer to the truth, in that $\lambda(N)\leq (\log N)^{o(1)}$. We have not tried to fully optimise the upper bound for $\lambda(N)$ in this paper, since the argument is quite technical even without such an optimisation. Nonetheless, we do not believe the methods in this paper alone are strong enough to prove $\lambda(N) \ll (\log N)^{1-c}$ for some $c>0$, let alone $\lambda(N)\ll (\log N)^{o(1)}$, for which significant new ideas would be needed.

 \subsection*{Formalisation of the proof} The proofs of the main results of this paper have now been completely formally verified using the Lean proof assistant, in joint work with Bhavik Mehta. We discuss this formalisation in Appendix~\ref{app:form}, which is written for the curious mathematician who has never themselves yet dabbled in formalistaion, to give a flavour of what the process involves.

 \subsection*{Comparison to Croot's method}
The main result of Croot \cite{Cr2003} is actually a density statement for sets of sufficiently smooth numbers. Roughly speaking, Croot proves that if $A\subset [2,N]$ has $\sum_{n\in A}\frac{1}{n}> 6$ and all $n\in A$ are `$n^{1/4-o(1)}$-smooth' (i.e. if a prime power $q$ divides $n\in A$ then $q\leq n^{1/4-o(1)}$) then $A$ contains some $S$ with $\sum_{n\in S}\frac{1}{n}=1$. This suffices to deduce the colouring statement of Theorem~\ref{th-croot}, but is too weak to deduce an unrestricted density statement, since the set of $n\in \bbn$ which are \emph{not} $n^{1/4-o(1)}$ smooth has positive density. 

In \cite{Crootphd} Croot suggests that it is possible that his method could be strengthened to improve this smoothness threshold to $n^{1-o(1)}$, which would suffice to show that all sets with positive density contain some $S$ with $\sum_{n\in S}\frac{1}{n}=1$, since almost all integers $n\in\bbn$ are $n^{1-o(1)}$ smooth. Such an improvement is essentially what we provide in this paper, and our methods are a stronger form of the ideas introduced by Croot. Many of our lemmas are similar to those of \cite{Cr2003}, but since we require stronger forms of them than Croot uses, and to keep this paper self-contained, we include full proofs throughout, indicating the correspondence to parts of \cite{Cr2003} as we proceed.

\subsection*{Structure of the paper}
In Section~\ref{sec2} we state our main technical proposition, Proposition~\ref{th-techmain}, and use elementary methods to deduce Theorems~\ref{th-dens} and \ref{th-quant}. The remainder of the paper is occupied with the proof of Proposition~\ref{th-techmain}.

Before we begin, we remark on the use of explicit constants in this paper. In many places we have used explicit constants and exponents (e.g. $(\log N)^{-1/100}$). The precise form of these exponents should not be taken particularly seriously -- in general we have chosen nice round numbers that allow for our statements to be true with `room to spare'.

\subsection*{Acknowledgements}
The author is funded by a Royal Society University Research Fellowship. We would like to thank the anonymous referee for suggesting the inclusion of expanded expository remarks after Proposition~\ref{th-techmain}. 

\section{The main technical proposition}\label{sec2}
We introduce the convenient notation $R(A)=\sum_{n\in A}\frac{1}{n}$\label{def-ra}, which will be used throughout the rest of the paper. A summary of the non-standard notation employed often in this paper, together with the definition and where it is first used, is in Table~\ref{tab-not}.

\renewcommand{\arraystretch}{1.5}
\begin{table}[h!]
\centering
\caption{Index of non-standard notation}
\label{tab-not}
\begin{tabularx}{1\textwidth} {
| >{\centering\arraybackslash\hsize=1\hsize}X | >{\centering\arraybackslash\hsize1.8\hsize}X | >{\centering\arraybackslash\hsize=.2\hsize}X |
}
 \hline
 $R(A)$ & $\sum_{n\in A}\frac{1}{n}$ & p.\pageref{def-ra}\\
 $A_q$ & $\{ n\in A: q\mid n\textrm{ and }(q,n/q)=1\}$ & p.\pageref{def-aq}\\
 $\mathcal{Q}_A$ & $\{ q\textrm{ a prime power}: A_q\neq \emptyset\}$ & p.\pageref{def-qa}\\
 $R(A;q)$ & $\sum_{n\in A_q}\frac{q}{n}$ & p.\pageref{def-raq}\\
 \hline
\end{tabularx}

\end{table}

 We recall two elementary estimates of Mertens which will be used frequently:
\begin{equation}\label{mertens1}
\sum_{q\leq X}\frac{1}{q}= \log\log X+c+O(1/\log X)\textrm{ for some constant }c,
\end{equation}
and
\begin{equation}\label{mertens2}
\prod_{p\leq X}\brac{1-\frac{1}{p}}^{-1}\asymp \log X,
\end{equation}
where $q$ indicates the summation is restricted to prime powers, and $p$ indicates the product is restricted to primes. Proofs of these estimates can be found in any textbook on analytic number theory (for example, \cite[Chapter 2]{MV}). Much more precise estimates are known, but these weak forms are all that we will require.

We now present the main technical proposition of this paper, from which the results in the introduction follow via elementary manipulations. Roughly speaking, this shows that a solution to $R(S)=1/d$, for some small $d$ in a restricted range, can be found in any set $A$ which is 
\begin{enumerate}
\item reasonably large (i.e. $R(A)$ is large),
\item not too `rough' (i.e. every $n\in A$ has at least two small divisors),
\item reasonably smooth, (i.e. no $n\in A$ is divisible by a prime power $q>n^{1-o(1)}$), and
\item arithmetically `regular' (i.e. every $n\in A$ has roughly $\log\log n$ (the expected amount) many distinct prime factors).
\end{enumerate}
Conditions (2)-(4) are not very restrictive in practice, since they are satisfied by almost every $n\in \bbn$, and hence in applications one can easily discard any potential troublemakers and still be left with a large enough $R(A)$. The fact that we are finding solutions to $R(S)=1/d$ with $d$ small is also not a problem in applications, since any $d$ such disjoint solutions can be trivially combined into some $S'$ with $R(S')=1$ as required.

This should be compared to the Main Theorem of \cite{Cr2003}, which achieves the conclusion $R(S)=1$ directly, but under a stronger smoothness hypothesis (namely a smoothness threshold of $N^{1/4-o(1)}$ rather than our $N^{1-o(1)}$), and without any roughness hypothesis (corresponding to our condition (2)). The convenience of every $n\in A$ having at least two small divisors is a novel feature of our approach.

\begin{proposition}\label{th-techmain}
Let $N$ be sufficiently large. Suppose $A\subset [N^{1-1/\log\log N},N]$ and $1\leq y\leq z\leq (\log N)^{1/500}$ are such that
\begin{enumerate}
\item $R(A)\geq 2/y+(\log N)^{-1/200}$,
\item every $n\in A$ is divisible by some $d_1$ and $d_2$ where $y\leq d_1$ and $4d_1\leq d_2\leq z$,
\item every prime power $q$ dividing some $n\in A$ satisfies $q\leq N^{1-6/\log\log N}$, and
\item every $n\in A$ satisfies
\[\tfrac{99}{100}\log\log N\leq \omega(n) \leq 2\log\log N.\]
\end{enumerate}
There is some $S\subset A$ such that $R(S)=1/d$ for some $d\in [y,z]$.
\end{proposition}

Before proceeding with the applications and proof of Proposition~\ref{th-techmain}, we make some further remarks about the reason we aim only to find a solution to $R(S)=1/d$ with $d$ small (instead of directly finding some $S$ with $R(S)=1$) and the utility of condition (2), that there exist some pair of small divisors. For simplicity we will only consider the case when $y=1$ and $z=4$, so that condition (2) becomes `every $n\in A$ is divisible by 4' and the conclusion becomes `there is some $S\subset A$ such that $R(S)=1/d$ for some $1\leq d\leq 4$'. 

Similar to Croot's method, we will would like to find solutions to $1=\sum_{n\in S}\frac{1}{n}$ using the circle method. Our formulation of the circle method (Proposition~\ref{prop-fourier} below) requires, amongst other assumptions, that
\begin{enumerate}
\item $R(A)\approx 2$ and
\item a further technical hypothesis about the distribution of multiples of elements of $A$ in short intervals.
\end{enumerate}
(Roughly speaking, the first ensures that the `major arc' contribution is positive, and the second ensures that the `minor arc' contribution is small.) Since condition (1) of Proposition~\ref{th-techmain} says $R(A)>2$ we can arrange for $R(A)\approx 2$ by discarding some elements of $A$. Our conditions are not strong enough to force the second, technical, hypothesis, however.

Following Croot, we aim to show that there is some subset $A'\subseteq A$ which satisfies this additional hypothesis via a combinatorial `pruning' technique. Unfortunately, to allow for the high `smoothness' threshold the additional hypothesis is quite restrictive, and therefore the pruning is severe. In particular, supposing we begin with $R(A)\approx 2$, our method shows that either $A$ itself satisfies the additional technical hypothesis, or else there is some pruned $B\subseteq A$ which does satisfy it -- but we have lost some of the mass of $A$, and can now only say that $R(B)\geq \frac{1}{3}R(A)\approx 2/3$. Since it is possible that $R(B)\approx 2/3$, whatever games we play we certainly cannot hope to find a solution to $R(S)=1$ with $S\subseteq B$! 

This means that we have to relax our initial goal, and instead be content with e.g. finding a solution to $R(S)=1/4$ (after pruning $B$ down a little further so that $R(B)\approx 2/4$ -- this is not strictly necessary in our simplified setup, but allowing for a small amount of further pruning eases some of the technical burdens in the general case). This introduces a further divisibility constraint however, since we can only hope to find a solution to $R(S)=1/4$ with $S\subseteq B$ if $4$ divides the lowest common multiple of $B$. This is where condition (2) becomes useful -- since every element of $A$ is divisible by $4$, certainly $4$ must divide the lowest common multiple of $B\subseteq A$. The universality of condition (2) is helpful here -- if there existed a subset of $A$, even a sparse one, that consisted of only odd numbers then it may be that our pruning procedure conspired so that $B$ is contained in this subset, and then we cannot find a solution to $R(S)=1/4$ with $S\subseteq B$. 

We hope that this sketch illustrates some of the issues involved and the reason for the particular form of Proposition~\ref{th-techmain}. 

In the remainder of this section we will deduce Theorems~\ref{th-dens} and \ref{th-quant} from Proposition~\ref{th-techmain}, which will be proved in the remainder of the paper.

\subsection{Proof of Theorem~\ref{th-quant}}

We deduce Theorem~\ref{th-quant} via the following corollary of Proposition~\ref{th-techmain}.
\begin{corollary}\label{cor-techmain}
Suppose $N$ is sufficiently large and $A\subset [N^{1-1/\log\log N},N]$ is such that
\begin{enumerate}
\item $R(A)\geq (\log N)^{1/200}$,
\item every $n\in A$ is divisible by some prime $p$ satisfying $5 \leq p \leq (\log N)^{1/500}$,
\item every prime power $q$ dividing some $n\in A$ satisfies $q\leq N^{1-6/\log\log N}$, and
\item every $n\in A$ satisfies
\[\tfrac{99}{100}\log\log N\leq \omega(n) \leq 2\log\log N.\]
\end{enumerate}
There is some $S\subset A$ such that $R(S)=1$.
\end{corollary}
\begin{proof}
Let $k$ be maximal such that there are disjoint $S_1,\ldots,S_k\subset A$ where, for each $1\leq i\leq k$, there exists some $d_i\in [1,(\log N)^{1/500}]$ such that $R(S_i)=1/d_i$. Let $t(d)$ be the number of $S_i$ such that $d_i=d$. If there is any $d$ with $t(d)\geq d$ then we are done, taking $S$ to be the union of any $d$ disjoint $S_j$ with $R(S_j)=1/d$. Otherwise,
\[\sum_i R(S_i)= \sum_{1\leq d\leq (\log N)^{1/500}} \frac{t(d)}{d}\leq (\log N)^{1/500},\]
and hence if $A'=A\backslash (S_1\cup\cdots \cup S_k)$ then $R(A')\geq (\log N)^{1/500}$. 

We may now apply Proposition~\ref{th-techmain} with $y=1$ and $z=(\log N)^{1/500}$ -- note that condition (2) of Proposition~\ref{th-techmain} follows from condition (2) of Corollary~\ref{cor-techmain} with $d_1=1$ and $d_2=p\in [5,(\log N)^{1/500}]$ some suitable prime divisor. Thus there exists some $S'\subset A'$ such that $R(S')=1/d$ for some $d\in [1,(\log N)^{1/500}]$, contradicting the maximality of $k$.
\end{proof}

The deduction of Theorem~\ref{th-quant} is now a routine exercise in analytic number theory. We will require the following simple application of sieve theory. 

\begin{lemma}\label{lem-sieve1}
Let $N$ be sufficiently large and $z,y$ be two parameters such that $\log N \geq z>y\geq 3$. If $X$ is the set of all those integers not divisible by any prime in $p\in [y,z]$ then
\[\abs{ X\cap [N,2N)}\ll \frac{\log y}{\log z}N.\]
\end{lemma}
\begin{proof}
Any sieve will suffice here, even the Sieve of Eratosthenes; for example, the formulation of the sieve of Eratosthenes as given in \cite[Theorem 3.1]{MV} yields 
\[\abs{X\cap [N,2N)} = \prod_{y\leq p\leq z}\brac{1-\frac{1}{p}}N+ O(2^z).\]
Mertens' estimate \eqref{mertens2} yields
\[\prod_{y\leq p\leq z}\brac{1-\frac{1}{p}}\asymp \frac{\log y}{\log z},\]
and the result follows. 
\end{proof}

\begin{proof}[Proof of Theorem~\ref{th-quant}]
Let $C\geq 2$ be an absolute constant to be chosen shortly, and for brevity let $\epsilon = \log\log\log N/\log\log N$, so that we may assume that $R(A)\geq C\epsilon \log N$. Since $\sum_{n\leq X}\frac{1}{n}\ll \log X$, if $A'=A\cap [N^\epsilon,N]$ we have (assuming $C$ is sufficiently large) $R(A')\geq \frac{C}{2}\epsilon\log N$. 

Let $X$ be those integers $n\in [1,N]$ not divisible by any prime $p\in [5,(\log N)^{1/1200}]$. Lemma~\ref{lem-sieve1} implies that, for any $x\geq \exp(\sqrt{\log N})$,  
\[\abs{X\cap[x,2x)}\ll \frac{x}{\log\log N}\]
and hence, by partial summation, 
\[\sum_{\substack{n\in X\\ n\in [\exp(\sqrt{\log N}),N]}}\frac{1}{n}\ll \frac{\log N}{\log\log N}.\]
Similarly, if $Y$ is the set of those $N\in [1,N]$ such that $\omega(n) <\frac{99}{100}\log\log N$ or $\omega(n)\geq \frac{101}{100}\log\log N$ then Tur\'{a}n's estimate (see, for example, \cite[Theorem 2.12]{MV})
\[\sum_{n\leq x}(\omega(n)-\log\log n)^2\ll x\log\log x\]
implies that $\abs{Y\cap [x,2x)}\ll x/\log\log N$ for any $N\geq x\geq \exp(\sqrt{\log N})$, and so
\[\sum_{\substack{n\in Y\\ n\in [\exp(\sqrt{\log N}),N]}}\frac{1}{n}\ll \frac{\log N}{\log\log N}.\]
In particular, provided we take $C$ sufficiently large, we can assume that $R(A'\backslash (X\cup Y))\geq \frac{C}{4}\epsilon \log N$, say. 

Let $\delta=1-1/\log\log N$, and let $N_i=N^{\delta^i}$, and $A_i=(A'\backslash (X\cup Y))\cap [N_{i+1},N_i]$. Since $N_i\leq N^{e^{-i/\log\log N}}$ and $A'$ is supported on $n\geq N^\epsilon$, the set $A_i$ is empty for $i> \log(1/\epsilon)\log\log N$, and hence by the pigeonhole principle there is some $i$ such that
\[R(A_i)\geq \frac{C}{8}\frac{\epsilon\log N}{(\log\log N)\log(1/\epsilon)}.\]
By construction, $A_i\subset [N_{i+1},N_i]\subset [N_i^{1-1/\log\log N_i},N_i]$, and every $n\in A_i$ is divisible by some prime $p$ with $5\leq p\leq (\log N)^{1/1200}\leq (\log N_i)^{1/500}$. Furthermore, every $n\in A_i$ satisfies $\omega(n)\geq \frac{99}{100}\log\log N\geq \frac{99}{100}\log\log N_i$ and $\omega(n)\leq \frac{101}{99}\log\log N\leq 2\log\log N_i$.

Finally, it remains to discard the contribution of those $n\in A_i$ divisible by some large prime power $q> N_i^{1-6/\log\log N_i}$. The contribution to $R(A_i)$ of all such $n$ is at most
\begin{align*}
\sum_{N_i^{1-6/\log\log N_i}< q\leq N_i}\sum_{\substack{n\leq N_i\\ q\mid n}}\frac{1}{n}
&\ll \sum_{N_i^{1-6/\log\log N_i}< q\leq N_i}\frac{\log(N_i/q)}{q}\\
&\ll \frac{\log N_i}{\log\log N_i}\sum_{N_i^{1-6/\log\log N_i}<q\leq N_i}\frac{1}{q}\\
&\ll \frac{\log N}{(\log\log N)^2},
\end{align*}
using Mertens' estimate \eqref{mertens1}. Provided we choose $C$ sufficiently large, this is $\leq R(A_i)/2$, and hence, if $A_i'\subset A_i$ is the set of those $n$ divisible only by prime powers $q\leq N_i^{1-6/\log\log N_i}$, then $R(A_i')\geq (\log N)^{1/200}$, say. All of the conditions of Corollary~\ref{cor-techmain} are now met, and hence there is some $S\subset A_i'\subset A$ such that $R(S)=1$, as required.
\end{proof}

\subsection{Proof of Theorem~\ref{th-dens}}

Finally, we show how Proposition~\ref{th-techmain} implies Theorem~\ref{th-dens}. We will require the following sieve estimate.

\begin{lemma}\label{sieve2}
Let $N$ be sufficiently large and $z,y$ be two parameters such that $(\log N)^{1/2}\geq z>4y\geq 8$. If $Y\subset [1,N]$ is the set of all those integers divisible by at least two distinct primes $p_1,p_2\in [y,z]$ where $4p_1<p_2$ then
\[\abs{ \{1,\ldots,N\}\backslash Y}\ll \brac{\frac{\log y}{\log z}}^{1/2}N.\]
\end{lemma}
\begin{proof}
Let $w\in (4y,z)$ be some parameter to be chosen later. Lemma~\ref{lem-sieve1} implies that the number of $n\in \{1,\ldots,N\}$ not divisible by any prime $p\in [w,z]$ is $\ll \frac{\log w}{\log z}N$.

Similarly, for any $p\in [w,z]$, the number of those $n\in [1,N]$ divisible by $p$ and no prime  $q\in [y,p/4)$ is
\[\ll \frac{\log y}{\log p}\frac{N}{p}.\]
It follows that the number of $n\in\{1,\ldots,N\}\backslash Y$ is 
\[\ll \brac{\frac{\log w}{\log z}+ \log y\sum_{p\geq w}\frac{1}{p\log p}}N.\]
By partial summation, $\sum_{p\geq w}\frac{1}{p\log p}\ll 1/\log w$, and hence 
\[\abs{ \{1,\ldots,N\}\backslash Y}\ll \brac{\frac{\log w}{\log z}+\frac{ \log y}{\log w}}N.\]
Choosing $w=\exp\brac{\sqrt{(\log y)(\log z)}}$ completes the proof. 
\end{proof}

\begin{proof}[Proof of Theorem~\ref{th-dens}]
Suppose $A\subset \bbn$ has upper density $\delta>0$. Let $y=C_1/\delta$ and $z=\delta^{-C_2\delta^{-2}}$, where $C_1,C_2$ are two absolute constants to be determined later. It suffices to show that there is some $d\in [y,z]$ and finite $S\subset A$ such that $R(S)=1/d$. Indeed, given such an $S$ we can remove it from $A$ and still have an infinite set of upper density $\delta$, so we can find another $S'\subset A\backslash S$ with $R(S')=1/d'$ for some $d'\in [y,z]$, and so on. After repeating this process at least $\lceil z-y\rceil^2$ times there must be some $d\in [y,z]$ with at least $d$ disjoint $S_1,\ldots,S_d\subset A$ with $R(S_i)=1/d$. Taking $S=S_1\cup\cdots \cup S_d$ yields $R(S)=1$ as required.

By definition of the upper density, there exist arbitrarily large $N$ such that $\abs{A\cap [1,N]}\geq \frac{\delta}{2}N$. The number of $n\in [1,N]$ divisible by some prime power $q\geq N^{1-6/\log\log N}$ is
\[\ll N \sum_{N^{1-6/\log\log N}<q\leq N}\frac{1}{q}\ll \frac{N}{\log\log N}\]
by Mertens' estimate \eqref{mertens1}. Further, by Tur\'{a}n's estimate
\[\sum_{n\leq N}(\omega(n)-\log\log N)^2 \ll N \log\log N,\]
the number of $n\in [1,N]$ that do not satisfy 
\begin{equation}\label{divs}
\tfrac{99}{100}\log\log N\leq \omega(n) \leq 2\log\log N
\end{equation}
is $\ll N/\log\log N$. Finally, provided we choose $C_2$ sufficiently large in the definition of $z$, Lemma~\ref{sieve2} ensures that the proportion of all $n\in \{1,\ldots,N\}$ not divisible by at least two distinct primes $p_1,p_2\in [y,z]$ with $4p_1<p_2$ is at most $\frac{\delta}{8}N$, say. 

In particular, provided $N$ is chosen sufficiently large (depending only on $\delta$), we may assume that $\abs{A_N}\geq \frac{\delta}{4}N$, where $A_N\subset A$ is the set of those $n\in A\cap [N^{1-1/\log\log N},N]$ which satisfy conditions (2)-(4) of Proposition~\ref{th-techmain}. Since $\abs{A_N}\geq \frac{\delta}{4}N$, 
\[R(A_N) \gg -\log(1-\delta/4)\gg \delta.\]
In particular, since $y=C_1/\delta$ for some suitably large constant $C_1>0$, we have that $R(A_N)\geq 4/y$, say. All of the conditions of Proposition~\ref{th-techmain} are now satisfied (provided $N$ is chosen sufficiently large in terms of $\delta$), and hence there is some $S\subset A_N\subset A$ such that $R(S)=1/d$ for some $d\in [y,z]$, which suffices to prove Theorem~\ref{th-dens} as discussed above. 
\end{proof}

It remains to establish Proposition~\ref{th-techmain}, which will be the task of the remainder of the paper.

\section{A Fourier analytic argument}
We follow the approach of Croot \cite{Cr2003}, and detect solutions to $R(S)=1/k$ for some fixed $k$ using Fourier analysis. The two important differences in the following proposition from the work of Croot are the detection of solutions to $R(S)=1/k$ for arbitrary integer $k$ (a flexibility which we require for our applications), and that the condition (4) is now weighted by a factor depending on $q$. By contrast, in \cite{Cr2003} Croot has the simpler condition that the number of $n\in A$ such that no element of $I$ is divisible by $n$ is at most $N^{3/4-o(1)}$. Roughly, our approach allows for this $N^{3/4-o(1)}$ to be replaced with $N^{o(1)}$, while replacing $n\in A$ with the weaker $n\in A_q$. It is ultimately this replacement of $N^{3/4-o(1)}$ with $N^{o(1)}$ that results in the `smoothness threshold' being $N^{1-o(1)}$ rather than Croot's $N^{1/4-o(1)}$, which in turn allows us to deduce strong density results.

For any finite $A\subset \bbn$ and prime power $q$ we define\label{def-aq}
\[A_q = \{ n\in A : q\mid n\textrm{ and }(q,n/q)=1\}\]
and let $\mathcal{Q}_A$ \label{def-qa} be the set of all prime powers $q$ such that $A_q$ is non-empty (i.e. those $p^r$ such that $p^r\| n$ for some $n\in A$). 

\begin{proposition}\label{prop-fourier}
There is an absolute constant $c>0$ such that the following holds. Suppose that $N\geq M\geq N^{3/4}$ and $k$ is an integer satisfying $1\leq k\leq cM$. Suppose further that $\eta\in(0,1)$ and $M/2\geq K\geq N^{3/4}$. Let $A\subset [M,N]$ be a set of integers such that
\begin{enumerate}
\item $R(A)\in [2/k-1/M,2/k)$,
\item $k$ divides the lowest common multiple of $A$,
\item if $q\in\mathcal{Q}_A$ then $q\leq c\min\brac{\frac{M}{k},\frac{\eta MK^2}{N^2(\log N)^2}}$,  and
\item for any interval $I$ of length $K$, either
\begin{enumerate}
\item \[\# \{ n\in A : \textrm{no element of }I\textrm{ is divisible by }n\}\geq M/\log N,\]
or
\item if $\mathcal{D}_I$ is the set of $q\in\mathcal{Q}_A$ such that
\[\#\{ n\in A_q: \textrm{no element of }I\textrm{ is divisible by }n\}<\eta \frac{M}{q},\]
there is some $x\in I$ divisible by all $q\in\mathcal{D}_I$.
\end{enumerate}
\end{enumerate}
There is some $S\subset A$ such that $R(S)=1/k$. (In fact, there are at least $2^{\Omega(\abs{A})}$ many such $S$.)
\end{proposition}

It may help to note that in our applications we will choose $\eta=N^{-o(1)}$, $k=N^{o(1)}$, and $M,K=N^{1-o(1)}$, and so the smoothness threshold is $N^{1-o(1)}$. 

\begin{proof}
For brevity, let $X= c\min(\frac{M}{k},\eta MK^2/N^2(\log N)^2)$, where $c>0$ is some absolute constant to be chosen later, so that condition (3) becomes $q\leq X$ for all $q\in\mathcal{Q}_A$. For any set of prime powers $\mathcal{P}$ we write $[\mathcal{P}]$ for the lowest common multiple of all $q\in\mathcal{P}$. Since $A$ is fixed, in this proof we just write $\mathcal{Q}=\mathcal{Q}_A$.

Let $F(A)$ count the number of subsets $S\subset A$ such that $kR(S)$ is an integer. Since $R(S)\leq R(A)<2/k$ and $R(S)=0$ if and only if $S=\emptyset$, the number of $S\subset A$ such that $R(S)=1/k$ is exactly $F(A)-1$.

We first note that, for any integers $a,b$, we have by orthogonality
\[1_{a/b\in \bbz} = \frac{1}{b}\sum_{-b/2< h\leq b/2}e\brac{\frac{ha}{b}},\]
where $e(x)=e^{2\pi ix}$. In particular, since for any $S\subset A$ the rational $kR(S)$ is of the form $km/[\mathcal{Q}]$ for some $m\in \bbz$, 
\[F(A)= \frac{1}{[\mathcal{Q}]}\sum_{-[\mathcal{Q}]/2< h\leq [\mathcal{Q}]/2}\prod_{n\in A}(1+e(kh/n)).\]
The most obvious contribution to $F(A)$ is from $h=0$, which contributes exactly $2^{\abs{A}}/[\mathcal{Q}]$. Furthermore, if $[\mathcal{Q}]$ is even, then $h=[\mathcal{Q}]/2$ contributes
\[\frac{1}{[\mathcal{Q}]}\prod_{n\in A}(1+e(k[\mathcal{Q}]/2n))\]
which is some real number $\geq 0$, since $e(k[\mathcal{Q}]/2n)\in \{-1,1\}$ for all $n$ dividing $[\mathcal{Q}]$ (and so certainly for all $n\in A$). Therefore,
\begin{equation}\label{fbound}
F(A)\geq \frac{2^{\abs{A}}}{[\mathcal{Q}]}+\frac{1}{[\mathcal{Q}]}\sum_{h\in J}\prod_{n\in A}(1+e(kh/n))
\end{equation}
where $J= (-[\mathcal{Q}]/2,[\mathcal{Q}]/2)\cap \bbz\backslash \{0\}$. Our treatment of $h\in J$ depends on whether $h$ is close to an integer of the form $t[\mathcal{Q}]/k$. We define the `major arc' corresponding to $t\in\bbz$ as
\[\mathfrak{M}(t) = \{ h\in J : \abs{h-t[\mathcal{Q}]/k}\leq K/2k\}\]
(note that since $[\mathcal{Q}] \geq \min(A)\geq K$ the major arcs $\mathfrak{M}(t)$ are disjoint for distinct $t\in\bbz$) and define the minor arcs as $\mathfrak{m} = J\backslash \cup_{t\in \bbz}\mathfrak{M}(t)$. Writing $h\in\mathfrak{M}(t)$ as $t[\mathcal{Q}]/k+r$, the contribution of $\mathfrak{M}(t)$ to the right-hand side of \eqref{fbound} is, using  $1+e(\theta)=2e(\theta/2)\cos(\pi \theta)$,
\[ \frac{1}{[\mathcal{Q}]}\sum_{\substack{r\in  [-K/2k,K/2k]\\ r\in J-t[\mathcal{Q}]/k} }\prod_{n\in A}(1+e(kr/n))=\frac{2^{\abs{A}}}{[\mathcal{Q}]}\sum_{\substack{r\in  [-K/2k,K/2k]\\ r\in J-t[\mathcal{Q}]/k} }\prod_{n\in A}\cos(\pi kr/n)e\brac{ \frac{kr}{2}R(A)}.\]
The total contribution from all major arcs is therefore
\[\frac{2^{\abs{A}}}{[\mathcal{Q}]}\sum_{r\in [-K/2k,K/2k]}\brac{\sum_{t}1_{r\in J-t[\mathcal{Q}]/k}}\prod_{n\in A}\cos(\pi kr/n)e\brac{ \frac{kr}{2}R(A)}.\]
(Note that since $k\mid [\mathcal{Q}]$ the first factor restricts $r$ to be an integer.) 

Since both $\cos(\pi kr/n)$ and $\sum_t 1_{r\in J-t[\mathcal{Q}]/k}$ are symmetric in $r$, this is equal to
\[\frac{2^{\abs{A}}}{[\mathcal{Q}]}\sum_{r\in [0,K/2k]}(2-1_{r=0})\brac{\sum_{t}1_{r\in J-t[\mathcal{Q}]/k}}\prod_{n\in A}\cos(\pi kr/n)\cos(\pi kr R(A)).\]
Since $n\geq M\geq K$ and $0\leq r\leq K/2k$ we have $kr/n\leq 1/2$ and so certainly $\cos(\pi kr/n)\geq 0$. Furthermore, we can write $kR(A)=2-\epsilon$ for some $0<\epsilon\leq k/M$ and hence (since $r$ is an integer)
\[\cos(\pi kr R(A)) = \cos(-\pi r\epsilon)\geq 0\]
for $0\leq r\leq K/2k$, and hence 
\[\prod_{n\in A}\cos(\pi kr/n)\cos(\pi kr R(A))\geq 0\]
for all $0\leq r\leq K/2k$. It follows that the contribution to the right-hand side of \eqref{fbound} from the union of all $\mathfrak{M}(t)$ is $\geq 0$. It therefore suffices to show that 
\begin{equation}\label{tobound}
\sum_{h\in\mathfrak{m}}\prod_{n\in A}\abs{\cos(\pi kh/n)}\leq 1/4,
\end{equation}
say. Indeed, this implies that $F(A)\geq 2^{\abs{A}-1}/[\mathcal{Q}]$. Since all prime powers $q\in\mathcal{Q}$ satisfy $q\leq X$ we have the bound
\[ [\mathcal{Q}]\leq  X^{\pi(X)}\leq e^{O(X)}\leq 2^{\abs{A}/2},\]
say, using Chebyshev's estimate that $\pi(X)\ll X/\log X$, the trivial estimate $\abs{A}\geq MR(A)\geq M/k$, and the fact that $X\leq cM/k$ (assuming $c$ sufficiently small). It follows that $F(A)\geq 2^{\abs{A}/2-1}>1$ as required.

For any $h\geq 0$, let $\mathcal{D}_h=\mathcal{D}_{I_h}$ as defined in the theorem statement, where $I_h$ is the interval of length $K$ centred at $kh$. That is, $\mathcal{D}_h$ is the set of those $q\in \mathcal{Q}$ such that 
\[\#\{ n\in A_q:\abs{h_n}> K/2 \}< \eta \frac{M}{q},\]
where $kh\equiv h_n \pmod{n}$ with $\abs{h_n}\leq n/2$.

We introduce the convenient notation $C(B;h) = \prod_{n\in B}\abs{\cos(\pi kh/n)}$. Let $\mathfrak{m}_1$ be those $h\in\mathfrak{m}$ such that (4a) holds for $I_h$, and $\mathfrak{m}_2=\mathfrak{m}\backslash\mathfrak{m}_1$. 

We first address those $h\in\mathfrak{m}_1$. Note that if $x\in [0,1/2]$ then
\[\cos(\pi x) \leq 1-x^2\leq  e^{-x^2}.\]
It follows that if $kh\equiv h_n \pmod{n}$, where $\abs{h_n}\leq n/2$, then
\[\abs{\cos(\pi kh/n)}\leq e^{-\frac{h_n^2}{n^2}}.\] 
Further, if $h\in\mathfrak{m}_1$, then $\abs{h_n}\geq K/2$ for at least $M/\log N$ many $n\in A$, and hence
\[C(A;h)\leq \exp\brac{-\sum_{n\in A}\frac{h_n^2}{n^2}}\leq \exp\brac{-\frac{K^2M}{4N^2\log N}}.\]
Therefore
\[\sum_{h\in\mathfrak{m}_1}C(A;h)\leq [\mathcal{Q}]\exp\brac{-\frac{K^2M}{4N^2\log N}}\leq e^{O(X)}\exp\brac{-\frac{K^2M}{4N^2\log N}}.\]
Provided $X\leq c K^2M/N^2\log N$ for some sufficiently small constant $c>0$, this is $\leq 1/8$, and hence to show \eqref{tobound} it suffices to show 
\begin{equation}\label{tobound3}
\sum_{h\in\mathfrak{m}_2}\prod_{n\in A}\abs{\cos(\pi kh/n)}\leq 1/8,
\end{equation}

We will show that, for any $h\in\mathfrak{m}_2$,
\begin{equation}\label{tobound2}
C(A;h)\leq N^{-4\abs{\mathcal{Q}\backslash\mathcal{D}_h}}.
\end{equation}
(Note that this is trivial when $\mathcal{D}_h=\mathcal{Q}$.) Before establishing \eqref{tobound2}, we show how it implies \eqref{tobound3}, and hence concludes the proof.

By condition (4b) of the hypotheses, since $h\in\mathfrak{m}_2$, $kh$ is distance at most $K/2\leq M/2$ from some multiple of $[\mathcal{D}_h]$. Therefore, for any $\mathcal{D}\subset \mathcal{Q}$, the number of $h\in\mathfrak{m}_2$ with $\mathcal{D}_h=\mathcal{D}$ is at most $M$ times the number of multiples of $[\mathcal{D}]$ in $[1,k[\mathcal{Q}]]$, which is at most
\[Mk\frac{[\mathcal{Q}]}{[\mathcal{D}]} \leq Mk\prod_{q\in\mathcal{Q}\backslash\mathcal{D}}q\leq kN^{\abs{\mathcal{Q}\backslash \mathcal{D}}+1}.\]
In particular, by \eqref{tobound2}, the contribution to the left-hand side of \eqref{tobound3} from all $h\in\mathfrak{m}_2$ with $\mathcal{D}_h=\mathcal{D}$ is at most $kN^{1-3\abs{\mathcal{Q}\backslash \mathcal{D}}}$. 

By definition of $\mathfrak{m}$, if $h\in\mathfrak{m}_2$ then $kh$ is distance greater than $K/2$ from any multiple of $[\mathcal{Q}]$, and hence $\mathcal{D}_h\neq \mathcal{Q}$. Therefore (using the trivial estimate $\abs{\mathcal{Q}}\leq N$)
\begin{align*}
\sum_{h\in\mathfrak{m}_2}C(A;h)
&\leq
kN\sum_{\mathcal{D}\subsetneq \mathcal{Q}}N^{-3\abs{\mathcal{Q}\backslash \mathcal{D}}}\\
&\leq \frac{k}{N}(1+1/N)^{\abs{\mathcal{Q}}}\\
&\ll k/N
\end{align*}
which is $\leq 1/8$ provided $c$ is sufficiently small, which proves \eqref{tobound3}.

It remains to establish \eqref{tobound2}. Using the trivial estimate $\omega(n)\ll \log n$, for any $n\in A$ there are $O(\log N)$ many $q\in\mathcal{Q}$ such that $n\in A_q$. It follows that 
 \[C(A;h)\leq \prod_{q\in \mathcal{Q}}C(A_q;h)^{\Omega(1/\log N)}.\]
To establish \eqref{tobound2}, therefore, it suffices to show that for every $q\in\mathcal{Q}\backslash\mathcal{D}_h$ we have $C(A_q;h)\leq N^{-C\log N}$ for some suitably large absolute constant $C>0$. 
 For any $q\in\mathcal{Q}\backslash \mathcal{D}_h$ there are, by definition of $\mathcal{D}_h$, at least $\eta M/q$ many $n\in A_q$ such that $\abs{h_n}>K/2$ and hence
\[C(A_q;h)\leq\exp\brac{-\sum_{n\in A_q}\frac{h_n^2}{n^2}}\leq \exp\brac{-\frac{\eta MK^2}{4N^2q}}.\] 
It therefore is enough to have that, for every $q\in\mathcal{Q}$, 
\[\frac{\eta MK^2}{4N^2q}\geq C(\log N)^2.\]
This follows from $q\leq X\leq c\eta MK^2/N^2(\log N)^2$, provided $c$ is small enough.
\end{proof}

\section{Technical lemmas}
In this section we assemble some useful technical lemmas, most of which are quantitatively stronger forms of lemmas used by Croot \cite{Cr2003}. The following is \cite[Lemma 2]{Cr2003}, which we reprove here to keep the paper self-contained.
\begin{lemma}\label{lem-basic}
If $0<\abs{n_1-n_2}\leq N$ then
\[\sum_{q\mid (n_1,n_2)}\frac{1}{q}\ll \log\log\log N,\]
where the summation is restricted to prime powers.
\end{lemma}
\begin{proof}
If $q\mid (n_1,n_2)$ then $q$ divides $\abs{n_1-n_2}$, and hence in particular $q\leq N$. The contribution of all prime powers $p^r$ with $r\geq 2$ is $O(1)$, and hence it suffices to show that $\sum_{p\mid \abs{n_1-n_2}}\frac{1}{p}\ll \log\log\log N$. Any integer $\leq N$ is trivially divisible by $O(\log N)$ many primes. Clearly summing $1/p$ over $O(\log N)$ many primes is maximised summing over the smallest $O(\log N)$ primes. Since there are $\gg (\log N)^{3/2}$ many primes $\leq (\log N)^2$, we have
\[\sum_{p\mid \abs{n_1-n_2}}\frac{1}{p}\ll \sum_{p\leq (\log N)^2}\frac{1}{p}\ll \log\log\log N\]
by Mertens' estimate \eqref{mertens1}.
\end{proof}
The following is a quantitatively stronger form of \cite[Lemma 3]{Cr2003}.
\begin{lemma}\label{lem-rtop}
Let $1/2>\epsilon>0$ and $N$ be sufficiently large, depending on $\epsilon$. If $A$ is a finite set of integers such that $R(A)\geq (\log N)^{-\epsilon/2}$ and $(1-\epsilon)\log\log N\leq \omega(n)\leq  2\log\log N$ for all $n\in A$ then 
\[\sum_{q\in\mathcal{Q}_A}\frac{1}{q} \geq (1-2\epsilon)e^{-1}\log\log N.\]
\end{lemma}
\begin{proof}
Since, by definition, every integer $n\in A$ can be written uniquely as $q_1\cdots q_t$ for $q_i\in \mathcal{Q}_A$ for some $t\in I = [(1-\epsilon)\log\log N, 2\log\log N]$, we have that, since $t!\geq (t/e)^t$, 
\[R(A)\leq  \sum_{t\in I}\frac{\brac{\sum_{q\in \mathcal{Q}_A}\frac{1}{q}}^t}{t!}\leq \sum_{t\in I}\brac{\frac{e}{t}\sum_{q\in \mathcal{Q}_A}\frac{1}{q}}^t.\]
Since $(ex/t)^t$ is decreasing in $t$ for $x<t$, either $\sum_{q\in \mathcal{Q}_A}\frac{1}{q}\geq (1-\epsilon)\log\log N$ (and we are done), or the summand is decreasing in $t$, and hence we have
\[(\log N)^{-\epsilon/2}\leq R(A)\leq 2\log\log N\brac{\frac{\sum_{q\in \mathcal{Q}_A}\frac{1}{q}}{(1-\epsilon)e^{-1}\log\log N}}^{(1-\epsilon)\log\log N}.\]
The claimed bound follows, using the fact that $e^{-\frac{\epsilon}{2(1-\epsilon)}}\geq 1-\epsilon$ for $\epsilon\in (0,1/2)$, choosing $N$ large enough such that $(2\log\log N)^{2/\log\log N}\leq 1+\epsilon^2$, say. 
\end{proof}
A weaker form of the following lemma is implicit in the proof of \cite[Proposition 3]{Cr2003}. For any finite $A\subset \bbn$ and prime power $q\in\mathcal{Q}_A$ we define\label{def-raq}
\[R(A;q) = \sum_{n\in A_q}\frac{q}{n}.\]
\begin{lemma}\label{lem-usingq}
There is a constant $c>0$ such that the following holds. Let $N\geq M\geq N^{1/2}$ be sufficiently large, and suppose that $1\leq k \leq c\log\log N$. Suppose that $A\subset [M,N]$ is a set of integers such that $\omega(n)\leq (\log N)^{1/k}$ for all $n\in A$. 

For all $q$ such that $R(A;q)\geq (\log N)^{-1/2}$ there exists $d$ such that
\begin{enumerate}
\item $qd > M\exp(-(\log N)^{1-1/k})$,
\item $\omega(d)\leq \tfrac{5}{\log k}\log\log N$, and
\item \[\sum_{\substack{n\in A_q\\qd\mid n\\ (qd,n/qd)=1}}\frac{qd}{n}\gg \frac{R(A;q)}{(\log N)^{2/k}}.\]
\end{enumerate}
\end{lemma}
\begin{proof}
Fix some $q$ with $R(A;q)\geq (\log N)^{-1/2}$. Let $D$ be the set of all $d$ such that if $p$ is a prime and $p^r \| d$ then 
\[p^r>y=\exp((\log N)^{1-2/k})\]
and 
\[qd\in (M\exp(-(\log N)^{1-1/k}),N].\]
We first claim that every $n\in A_q$ is divisible by some $qd$ with $d\in D$, such that $(qd,n/qd)=1$. This can be done greedily, just removing from $n/q$ all those prime power divisors $p^r\| n/q$ such that $p^r\leq y$, which removes at most 
\[y^{\omega(n)}\leq  \exp((\log N)^{1-1/k}).\]
We can therefore bound
\[R(A;q) \leq \sum_{d\in D}\frac{1}{d}\sum_{\substack{n\in A_q\\ qd\mid n\\ (qd,n/qd)=1}}\frac{qd}{n}.\]

We will control the contribution from those $d$ with $\omega(d)>\omega_0= \frac{5}{\log k}\log\log N$ with the trivial bound
\[\sum_{\substack{n\in A_q\\ qd\mid n\\ (qd,n/qd)=1}}\frac{qd}{n} \leq \sum_{\substack{n\leq N\\ qd\mid n}}\frac{qd}{n}\ll \log N\]
and Mertens' bound \eqref{mertens1}. Together these imply
\begin{align*}
\sum_{\substack{d\in D\\ \omega(d)>\omega_0}}\frac{1}{d}\sum_{\substack{n\in A_q\\ qd\mid n}}\frac{qd}{n}
&\ll \log N\sum_{\substack{d\\ p^r\| d\implies y<p^r\leq N\\ \omega(d)\geq \omega_0}} \frac{1}{d}\\
&\ll
k^{-\omega_0}\log N\sum_{\substack{d\\ p^r\| d\implies y<p^r\leq N}} \frac{k^{\omega(d)}}{d}\\
&\ll C_1^kk^{-\omega_0}\log N\prod_{y<p\leq N}(1+\frac{k}{p-1}) \\
 &\leq k^{-\omega_0}\log N\brac{C_2\frac{\log N}{\log y}}^{k}
\end{align*}
for some absolute constants $C_1,C_2>0$. Recalling the definitions of $y$ and $\omega_0$, this is
\[\leq C_2^kk^{-\omega_0}(\log N)^3\leq 1/\log N,\]
say, for $N$ sufficiently large. It follows that
\[\tfrac{1}{2}R(A;q)\leq \sum_{\substack{d\in D\\ \omega(d)\leq \omega_0}}\frac{1}{d}\sum_{\substack{n\in A_q\\ qd\mid n\\ (qd,n/qd)=1}}\frac{qd}{n} .\]
The result follows since 
\[\sum_{d\in D}\frac{1}{d}\leq \sum_{\substack{d\\ p^r\| d\implies y<p^r\leq N}}\frac{1}{d} \ll \prod_{y<p\leq N}\brac{1-\frac{1}{p-1}}^{-1}\ll \frac{\log N}{\log y}\ll (\log N)^{2/k}.\]
\end{proof}

\section{An iterative procedure}
To prove Proposition~\ref{th-techmain} we need to find some suitable $A'\subset A$ which satisfies the hypotheses of Proposition~\ref{prop-fourier} (for some suitable $k$). This is achieved by the following proposition, assuming various technical conditions, which we can arrange to hold via an iterative procedure. The following proposition is analogous to \cite[Proposition 3]{Cr2003}, and the basic structure of the proof is the same, but there are several technical refinements that are necessary in our proof.

\begin{proposition}\label{prop-tech}
Suppose $N$ is sufficiently large and $N\geq M\geq N^{1/2}$, and suppose that $A\subset [M,N]$ is a set of integers such that
\[\tfrac{99}{100}\log\log N\leq \omega(n)\leq  2\log\log N\quad\textrm{for all}\quad n\in A,\]
\[R(A)\geq (\log N)^{-1/101}\]
and, for all $q\in \mathcal{Q}_A$,
\[R(A;q) \geq (\log N)^{-1/100}.\]
Then either
\begin{enumerate}
\item there is some $B\subset A$ such that $R(B)\geq \tfrac{1}{3}R(A)$ and 
\[\sum_{q\in \mathcal{Q}_{B}}\frac{1}{q}\leq \frac{2}{3}\log\log N,\]
or
\item for any interval of length $\leq MN^{-2/(\log \log N)}$, either
\begin{enumerate}
\item \[\# \{ n\in A : \textrm{no element of }I\textrm{ is divisible by }n\}\geq M/\log N,\]
or
\item if $\mathcal{D}_I$ is the set of $q\in\mathcal{Q}_A$ such that
\[\#\{ n\in A_q: \textrm{no element of }I\textrm{ is divisible by }n\}<\frac{M}{2q(\log N)^{1/100}},\]
there is some $x\in I$ divisible by all $q\in\mathcal{D}_I$.
\end{enumerate}
\end{enumerate}
If $\sum_{q\in\mathcal{Q}_A}\frac{1}{q}\leq \frac{2}{3}\log\log N$ then case (2) is guaranteed. 
\end{proposition}
\begin{proof}
Let $I$ be any interval of length $\leq MN^{-2/(\log\log N)}$, and let $A_I$ be those $n\in A$ that divide some element of $I$. We may assume that $\abs{A\backslash A_I}< M/\log N$ (or else 2(a) holds), and we need to show that either there is some $x\in I$ divisible by all $q\in\mathcal{D}_I$, or the first case holds. 

Let $\mathcal{E}_I$ be the set of those $q\in\mathcal{Q}_A$ such that $R(A_I;q)> 1/2(\log N)^{1/100}$. For every $q\in \mathcal{D}_I$, by definition, 
\[R(A_I; q) \geq R(A;q) - \brac{\frac{M}{2q(\log N)^{1/100}}}\frac{q}{M}>\frac{1}{2(\log N)^{1/100}},\]
and hence in particular $\mathcal{D}_I\subset \mathcal{E}_I$.

For any $q\in \mathcal{E}_I$ we may therefore apply Lemma~\ref{lem-usingq} with $A$ replaced by $A_I$, and $k$ chosen such that $(\log N)^{1/k}=2\log\log N$. This produces some $d_q$ such that $qd_q>\abs{I}$ and $\omega(d_q)<\tfrac{1}{500}\log\log N$ (provided $N$ is sufficiently large), and
\[\sum_{\substack{n\in A_I\\ qd_q\mid n\\ (qd_q,n/qd_q)=1}}\frac{qd_q}{n}\gg \frac{1}{(\log N)^{1/100}(\log\log N)^2}.\]
By definition of $A_I$, every $n\in A_I$ with $qd_q\mid n$ must divide some $x\in I$ -- in fact, they must all divide the same $x\in I$ (call this $x_q\in I$, say), since all such $x$ are in particular divisible by $qd_q>\abs{I}$, which can divide at most one element in $I$. 

Let 
\[A_I^{(q)}=\{ n/qd_q : n\in A_I\textrm{ with }qd_q\mid n\textrm{ and }(qd_q,n/qd_q)=1\}\]
so that, assuming $N$ is sufficiently large, $R(A_I^{(q)})\geq (\log N)^{-1/99}$, say. We may therefore apply Lemma~\ref{lem-rtop} with $\epsilon=2/99$ (note that since $\omega(n)\geq \frac{99}{100}\log\log N$ for $n\in A$ and $\omega(d_q)<\frac{1}{500}\log\log N$, we must have $\omega(m)\geq \frac{97}{99}\log\log N$ for all $m\in A_I^{(q)}$). This implies that 
\[\sum_{r\in \mathcal{Q}_{A_I^{(q)}}}\frac{1}{r}\geq \frac{95}{99}e^{-1}\log\log N.\]
Trivially, $\mathcal{Q}_{A_I^{(q)}}\subset \mathcal{Q}_A$, and further by choice of $x_q$, all $r\in \mathcal{Q}_{A_I^{(q)}}$ divide $x_q$, and hence 
\[\sum_{\substack{r\mid x_q\\ r\in \mathcal{Q}_A}}\frac{1}{r}\geq \frac{95}{99}e^{-1}\log\log N\geq 0.35\log\log N.\]
For any two $n_1\neq n_2\in I$, we have
\[\sum_{q\mid (n_1,n_2)}\frac{1}{q}\ll \log\log\log N\leq 0.01\log\log N\]
for $N$ sufficiently large, by Lemma~\ref{lem-basic}. It follows that if $\sum_{q\in\mathcal{Q}_A}\frac{1}{q}\leq \frac{2}{3}\log\log N$ then there can be at most one such possible value for $x_q\in I$ as $q$ ranges over $\mathcal{E}_I$, and hence this common shared value of $x_q$ is an $x\in I$ divisible by all $q\in\mathcal{E}_I$, and hence certainly by all $q\in\mathcal{D}_I$, as required.
 
Furthermore, since $\sum_{q\in\mathcal{Q}_A}\frac{1}{q}\leq (1+o(1))\log\log N\leq 1.01\log\log N$, say, there must always be at most two distinct values of $x_q\in I$ as $q$ ranges over $\mathcal{E}_I$. If there is no $x\in I$ divisible by all $q\in\mathcal{D}_I$, there must be exactly two such values, say $w_1$ and $w_2$. 

Let $A^{(i)}=\{n\in A: n\mid w_i\}$ and $A^{(0)}=A\backslash (A^{(1)}\cup A^{(2)})$. Since every $q\in\mathcal{Q}_{A^{(1)}}$ divides $w_1$, 
\[\sum_{q\in \mathcal{Q}_{A^{(1)}}}\frac{1}{q}\leq \sum_{q\leq N}\frac{1}{q}- \sum_{q\mid w_2}\frac{1}{q}+ \sum_{q\mid (w_1,w_2)}\frac{1}{q}\leq (1-\tfrac{95}{99}e^{-1}+o(1))\log\log N.\]
For large enough $N$, the right-hand side is $\leq \frac{2}{3}\log\log N$, and similarly for $A^{(2)}$. Since $R(A^{(0)})+R(A^{(1)})+R(A^{(2)})\geq R(A)$, we are in the first case choosing $B=A^{(1)}$ or $B=A^{(2)}$, unless $R(A^{(0)})\geq R(A)/3$. In this latter case we will derive a contradiction.

Let $A'\subset A^{(0)}$ be the set of those $n\in A_I\cap A^{(0)}$ such that if $n\in A_q$ then $q\in\mathcal{E}_I$. By definition of $\mathcal{E}_I$ and Mertens' estimate \eqref{mertens1}, 
\[R(A^{(0)}\backslash A')\leq \frac{\abs{A\backslash A_I}}{M}+\sum_{q\in\mathcal{Q}_A\backslash \mathcal{E}_I}\frac{1}{q}R(A_I;q)\ll \frac{\log\log N}{(\log N)^{1/100}},\]
and so  in particular, since $R(A)\geq (\log N)^{-1/101}$, we have $R(A')\gg (\log N)^{-1/101}$.

In particular, $\abs{A'}\gg M/(\log N)^{-1/101}$. Therefore there must exist some $x\in I$ (necessarily $x\neq w_1$ and $x\neq w_2$ since $A'\subset A^{(0)}$) such that, if $A''=\{ n\in A' : n\mid x\}$, then
\[\abs{A''}\gg N^{2/\log\log N}(\log N)^{-1/101},\]
and hence $\abs{A''}\geq N^{3/2\log\log N}$, say. 

However, if $n\in A''$ then both $n\mid x$ and $n\mid w_1w_2$ (since every $q$ with $n\in A_q$ is in $\mathcal{E}_I$ and so divides either $w_1$ or $w_2$), and hence $n$ divides
\[(x,w_1w_2)\leq (x,w_1)(x,w_2)\leq \abs{x-w_1}\abs{x-w_2}\leq N^2.\]
Therefore the size of $A''$ is at most the number of divisors of some fixed integer $m\leq N^2$, which is at most $N^{(1+o(1))2\log 2/\log\log N}$ (see \cite[Theorem 2.11]{MV}, for example), and hence we have a contradiction for large enough $N$, since $2\log 2< 3/2$.
\end{proof}

Finally, to apply Proposition~\ref{prop-fourier} we need to prepare our set $A$ so that $R(A)\in [2/k-1/M,2/k)$ and $R(A;q)\geq (\log N)^{-1/100}$, say. Both can be achieved using a greedy `pruning' approach, adapted from the proof of \cite[Proposition 2]{Cr2003}.

\begin{lemma}\label{lem-pisqa}
Let $N$ be sufficiently large and $A\subset [1,N]$. There exists $B\subset A$ such that 
\[R(B)\geq  R(A)-\frac{1}{(\log N)^{1/200}}\]
and $R(B;q)\geq 2/(\log N)^{1/100}$ for all $q\in \mathcal{Q}_B$. 
\end{lemma}
\begin{proof}
We construct a sequence of decreasing sets $A=A_0\supsetneq A_1\supsetneq\cdots \supsetneq A_i$ as follows. Given some $A_i$, if there is a prime power $q_i\in\mathcal{Q}_{A_i}$ such that 
\[R(A_i;q_i)< \frac{2}{(\log N)^{1/100}},\]
then we let $A_{i+1}=A_i\backslash (A_i)_{q_i}$. If no such $q_i$ exists then we halt the construction. This process must obviously terminate in some finite time (since some non-empty amount of $A_i$ is being removed at each step). Suppose that it halts at $A_j=B$, say. The amount lost from $R(A)$ at step $i$ is
\[\sum_{n\in (A_i)_{q_i}}\frac{1}{n}=\frac{1}{q_i}R(A_i;q_i)< \frac{2}{q_i(\log N)^{1/100}},\]
and furthermore each $q\leq N$ can appear as at most one such $q_i$, since after removing $(A_i)_{q_i}$ anything left in $A_i$ cannot have $q_i$ as a coprime divisor. It follows that
\[R(B)> R(A) -\frac{2}{(\log N)^{1/100}}\sum_{q\leq N}\frac{1}{q}\geq R(A)-\frac{1}{(\log N)^{1/200}},\]
since $\sum_{q\leq N}\frac{1}{q}\ll \log\log N$.
\end{proof}
\begin{lemma}\label{lem-pisq}
Suppose that $N$ is sufficiently large and $N\geq M\geq N^{1/2}$. Let $\alpha > 2/(\log N)^{1/200}$ and $A\subset [M,N]$ be a set of integers such that
\[R(A) \geq \alpha+\frac{1}{(\log N)^{1/200}}\]
and if $q\in\mathcal{Q}_A$ then $q\leq M/(\log N)^{1/100}$.

There is a subset $B\subset A$ such that $R(B)\in [\alpha-1/M,\alpha)$ and, for all $q\in \mathcal{Q}_B$, 
\[R(B;q)\geq \frac{1}{(\log N)^{1/100}}.\]
\end{lemma}
\begin{proof}
We first apply Lemma~\ref{lem-pisqa} to produce some $A'\subset A$ such that $R(A')\geq \alpha$ and $R(A';q) \geq 2/(\log N)^{1/100}$ for all $q\in\mathcal{Q}_{A'}$. 

We now argue that whenever $D$ is such that $R(D)\geq \alpha$ and $R(D;q) \geq (\log N)^{-1/100}$ for all $q\in \mathcal{Q}_D$ there exists some $x\in D$ such that $R(D\backslash \{x\};q)\geq (\log N)^{-1/100}$ for all $q\in\mathcal{Q}_{D}$. Given this, the lemma immediately follows, since we can continue removing such elements from $A'$ one at time until $R(B)$ falls in the required interval. 

To see why the above fact holds, apply Lemma~\ref{lem-pisqa} to obtain some $B\subset D$ (such that $R(B)\geq (\log N)^{-1/200}$, and hence in particular $B$ is non-empty), and let $x$ be any element of $B$. If $x\not\in D_q$ then by definition $R(D\backslash\{x\};q)=R(D;q)\geq (\log N)^{-1/100}$. If $x\in D_q$ then $x\in B_q$, and so
\[R(D\backslash\{x\};q)\geq R(B;q)-\frac{q}{x}\geq \frac{2}{(\log N)^{1/100}}-\frac{q}{M}\geq \frac{1}{(\log N)^{1/100}}\]
as required.
\end{proof}

We now have everything we need to prove Proposition~\ref{th-techmain}. The idea is, after pruning the set into a suitable form by repeated applications of Lemma~\ref{lem-pisq}, an application of Proposition~\ref{prop-tech} ensures that the hypotheses of Proposition~\ref{prop-fourier} are satisfied, which in turn delivers a solution to $R(S)=1/d$ for some suitable $d$ as required.

\begin{proof}[Proof of Proposition~\ref{th-techmain}]
Let $M=N^{1-1/\log\log N}$ and $d_i = \lceil y \rceil +i$. By repeated applications of Lemma~\ref{lem-pisq} we can find a sequence $A\supset A_0\supset A_1\supset\cdots \supset A_{t}$, where $d_t=\lceil z/4\rceil-1$, such that
\[R(A_i)\in [2/d_i-1/M,2/d_i)\quad\textrm{ and }\quad R(A_i;q)\geq (\log N)^{-1/100}\textrm{ for all }q\in \mathcal{Q}_{A_i}.\]
(Note that the hypotheses of Lemma~\ref{lem-pisq} continue to hold since 
\[\frac{2}{d_i}-\frac{1}{M}\geq \frac{2}{d_{i}+1}+\frac{1}{(\log N)^{1/200}}\geq \frac{3}{(\log N)^{1/200}}\]
for all $0\leq i\leq t$.) Let $0\leq j\leq t$ be minimal such that there is a multiple of $d_j$ in $A_j$. Such a $j$ exists by assumption, since every $n\in A$ is divisible by some $d\in[y,z/4)$. 

Suppose first that case (2) of Proposition~\ref{prop-tech} holds for $A_j$. The hypotheses of Proposition~\ref{prop-fourier} are now met with $k=d_j$, $\eta=1/2(\log N)^{1/100}$, and $K=MN^{-2/\log \log N}$. This yields some $S\subset A'\subset A$ such that $R(S)=1/d_j$ as required.

Otherwise, Proposition~\ref{prop-tech} yields some $B\subset A_j$ such that 
\[R(B)\geq 2/3d_j-1/M\geq 1/2d_j+(\log N)^{-1/200}\]
and where $\sum_{q\in\mathcal{Q}_B}\frac{1}{q}\leq \frac{2}{3}\log\log N$. Let $e_i = 4d_j+i$ and, once again, repeatedly apply Lemma~\ref{lem-pisq} to find a sequence $B\supset B_0\supset \cdots\supset B_r$, where $e_r=\lfloor z\rfloor$, such that
\[R(B_i)\in [2/e_i-1/M,2/e_i)\quad\textrm{ and }\quad R(B_i;q)\geq (\log N)^{-1/100}\textrm{ for all }q\in \mathcal{Q}_{B_i}.\]
By minimality of $j$, no $d\in [y,d_j)$ divides any element of $A_j$, and hence every $n\in A_j$ is divisible by some $e\in [4d_j,z]$. In particular, there must exist some $0\leq s\leq r$ such that $B_s$ contains a multiple of $e_s$. Furthermore, since $\sum_{q\in \mathcal{Q}_{B_s}}\frac{1}{q}\leq \sum_{q\in\mathcal{Q}_B}\frac{1}{q}\leq \frac{2}{3}\log\log N$ we must be in the second case of Proposition~\ref{prop-tech}. The hypotheses of Proposition~\ref{prop-fourier} are now met with $k=e_s$ and $\eta,K$ as above, and thus there is some $S\subset B_j\subset A$ such that $R(S)=1/e_s$. 
\end{proof}

\appendix

\section{Pomerance's lower bound}

Recall that $\lambda(N)$ is the maximum of $\sum_{n\in A}\frac{1}{n}$ as $A$ ranges over those $A\subset \{1,\ldots,N\}$ with no subsets $S\subset A$ such that $\sum_{n\in S}\frac{1}{n}=1$. In this appendix we present an observation of Pomerance, as related in Croot's PhD thesis \cite{Crootphd}, which yields the best lower bound for $\lambda(N)$ that we know of.

\begin{theorem}
\[\lambda(N) \gg (\log\log N)^2.\]
\end{theorem}
\begin{proof}
Let $C\geq 1$ be some absolute constant, to be chosen later, and let $A$ be the set of all those $n\in \{1,\ldots,N\}$ such that, if $p$ is the largest prime divisor of $n$, then $p\log p> Cn$. Considering the contribution to the sum from all those $n$ with a fixed largest prime divisor separately, we see that 
\begin{align*}
\sum_{n\in A}\frac{1}{n}
&\geq \sum_{p\leq N/\log N}\frac{1}{p}\sum_{m\leq \log p/C}\frac{1}{m}\\
&\gg \sum_{p\leq N/\log N}\frac{\log\log p}{p}\\
&\gg (\log\log N)^2.
\end{align*}
It remains to prove that there are no distinct $n_1,\ldots,n_k\in A$ such that 
\[\frac{1}{n_1}+\cdots+\frac{1}{n_k}=1.\]
Suppose otherwise, and fix such $n_1,\ldots,n_k$. Let $p$ be the largest prime which divides any of $n_1,\ldots,n_k$, and suppose (without loss of generality) $n_1=pm_1<\cdots<n_l=pm_l$ are all those $n_i$ divisible by $p$. Then
\[\frac{1}{p}\brac{\frac{1}{m_1}+\cdots+\frac{1}{m_l}}=1-\frac{1}{n_{l+1}}-\cdots-\frac{1}{n_k},\]
where the right-hand side is some fraction whose denominator is not divisible by $p$. Therefore $p$ must divide the numerator of $\frac{1}{m_1}+\cdots+\frac{1}{m_l}$, and hence $p$ divides
\[[m_1,\ldots,m_l]\brac{\frac{1}{m_1}+\cdots+\frac{1}{m_l}},\]
where $[m_1,\ldots,m_l]$ is the lowest common multiple. Chebyshev's estimate shows that $[1,\ldots,m]\leq e^{O(m)}$ for all large enough $m$, and hence
\[p\leq [1,\ldots,m_l]\brac{1+\frac{1}{2}+\cdots+\frac{1}{m_l}}\ll (\log m_l) e^{O(m_l)}\leq e^{O(m_l)}.\]
It follows that $\log p\leq O(m_l)$, and so $n_l=pm_l\gg p\log p$. By construction, however, $p$ is the largest prime divisor of $n_l$, and hence $p\log p>Cn_l$, which is a contradiction for a suitable choice of $C$.
\end{proof}

\section{Formalisation of the proof}\label{app:form}
\smallskip
\begin{center}by \textsc{Thomas F. Bloom} and \textsc{Bhavik Mehta}\end{center}
\medskip

All proofs contained in this paper, in particular the proofs of Theorem~\ref{th-dens}, have been formally verified using the Lean proof assistant \cite{MouraLean}, in joint work with Bhavik Mehta. This verification is complete, in that all prerequisites have also been verified, meaning that the proofs of the main results of this paper have been formally checked down to an axiomatic level. The Lean code for this verification is freely available at \url{https://github.com/b-mehta/unit-fractions}. We use many prerequisite results from the Lean mathematical library \mathlib\ \cite{mathlib}, a central repository of commonly used mathematical results and definitions which is developed by a large volunteer community. This repository can be explored at \url{https://leanprover-community.github.io/mathlib_docs/index.html}.

\subsection{The main results}
We first present the actual Lean statements which have been formally proved; in some Lean projects it is a non-trivial task to verify that the Lean theorems do indeed correspond to the natural language theorems, and that one has actually proved what was intended. Fortunately, the relatively elementary nature of the statements of Theorems~\ref{th-dens} and \ref{th-quant} make this straightforward in this instance. 

The following is the Lean statement of Theorem~\ref{th-dens}:
\begin{lstlisting}
theorem unit_fractions_upper_density (A : set ℕ) : 
    upper_density A > 0 → 
    ∃ (S : finset ℕ), S ⊆ A ∧ ∑ n in S, (1 / n : ℚ) = 1
\end{lstlisting}
This is proved on line 1261 of \texttt{final\textunderscore results.lean}. This statement is very close to its natural language analogue, after one checks that \texttt{upper\textunderscore density} means what it should. The statements \lstinline{(A : set ℕ)} and \lstinline{(S : finset ℕ)} are type declarations, informing Lean that $A$ and $S$ are a set and finite set respectively of natural numbers. 

Note that we require further type information (called a `coercion') when we write \lstinline{(1 / n : ℚ)} instead of simply \lstinline{1/n} -- this is because since \lstinline{n} has type \lstinline{ℕ}, by default Lean will interpret \lstinline{1/n} as having type \lstinline{ℕ} also, performing a `natural number division' where \lstinline{1/n} is 0 for $n > 1$. To avoid this we use an explicit coercion to $\mathbb{Q}$ so that Lean knows we mean the rational $1/n$. 

The following is the Lean statement of Theorem~\ref{th-quant} (which is proved on line 2042 of \texttt{final\textunderscore results.lean}). 
\begin{lstlisting}
theorem unit_fractions_upper_log_density : ∃ C : ℝ,
  ∀ᶠ (N : ℕ) in at_top, ∀ A ⊆ Icc 1 N,
  C * log(log(log N)) / log(log N) * log N ≤ ∑ n in A, (1 / n : ℚ) 
    → ∃ S ⊆ A, ∑ n in S, (1 / n : ℚ) = 1
\end{lstlisting}
This statement requires a little more interpretation. The statement \lstinline{∀ᶠ (N : ℕ) in at_top} is a Lean formalisation of the statement `for all sufficiently large $N\in\mathbb{N}$'. We write \lstinline{Icc 1 N} for $[1,N]\cap \mathbb{N}$ (the \lstinline{Icc} stands for `Interval closed closed'). Otherwise the statement is as in Theorem~\ref{th-quant}. 

Of course, every proposition and lemma in this paper also has a corresponding formal statement. Indeed, the process of formalisation made it natural to divide a single proof in this paper into multiple smaller claims, so the number of lemma statements is much greater in the formalisation. To give another example of what the translation from natural language mathematics to Lean looks like, the following is the formally verified Lean version of Lemma~\ref{lem-rtop} (line 1659 of \texttt{aux\textunderscore lemmas.lean}).

\begin{lstlisting}
lemma rec_qsum_lower_bound (ε : ℝ) (hε1 : 0 < ε) (hε2 : ε < 1/2) :
  ∀ᶠ (N : ℕ) in at_top, ∀ A : finset ℕ,
  log N ^ (-ε / 2) ≤ rec_sum A → 
  (∀ n ∈ A, (1 - ε) * log (log N) ≤ (ω n : ℝ)) → 
  (∀ n ∈ A, (ω n : ℝ) ≤ 2 * log (log N)) → 
  (1 - 2*ε) * exp (-1) * log (log N) ≤ 
        ∑ q in ppowers_in_set A, (1 / q : ℝ)
\end{lstlisting}

Again, the reader should verify for their satisfaction that this does indeed state the same thing as Lemma~\ref{lem-rtop}. The code \lstinline{rec_sum A} is what is called $R(A)$ in the paper, and \lstinline{ppowers_in_set A} is what we have called $\mathcal{Q}_A$. The coercion \lstinline{(ω n : ℝ)} is required since \lstinline{ω n} is defined as a function from $\mathbb{N}\to\mathbb{N}$, and we wish to compare it in an inequality with an element of $\mathbb{R}$.
\subsection{The proof}

In total the proof is spread out over five files with a total of 9255 lines of code. The total amount of Lean code written for this project is larger than this, since there were several classic elementary number theoretic results (for example Mertens' estimates and Chebyshev's estimate) that were also required. These more foundational results will be included in \texttt{mathlib}, the central repository of Lean results. 

The formal proofs follow very similar lines to the proofs as presented in this paper, although of course are greatly expanded in length, as every claim, no matter how obvious or elementary, must be explicitly written and checked. The formal verification of some of these `obvious' facts was perhaps the most tedious part of the formalisation process. In particular, the `asymptotic' type of quantitative estimates frequently implicitly used in this paper, such as the fact that 

\[\frac{1}{\log N}+\frac{1}{2(\log N)^{1/100}}\brac{\frac{501}{500}\log\log N}\leq \frac{1}{6(\log N)^{1/101}},\]
for all sufficiently large $N$, are obvious after a moment of thought to any mathematically mature human reader, but required a lengthy formal proof.

One difference between this paper and the formalisation is that the `algorithmic' style or argument frequently used in this paper, such as that in the proof of Lemma~\ref{lem-pisqa}, was converted into a more formal inductive scheme. Although in principle Lean can formalise a proof phrased as an algorithm that terminates in finite time, checking an inductive hypothesis is much simpler and easier to structure. All of the algorithmic style proofs in this paper were translated into suitable inductive analogues (although in some cases the best inductive hypothesis to use for the simplest proof required some thought).
\subsection{The process}

The formalisation began in January 2022 and concluded in July 2022. At the beginning of the formalisation, the first author had no experience with Lean at all, and learnt Lean (or at least a sufficient subset of Lean) through the formalising process. 

A very helpful intermediate step in the formalisation was the production of a `blueprint' of the paper, which contained expanded versions of all statements and proofs used in this paper, including many `obvious' steps that were omitted as is customary in the mathematical literature. Some proofs were also reorganised with a view to what would be the simplest order to formalise things in. This part of the process required no knowledge of Lean, but was a necessary prerequisite to allow for a smooth translation from natural language proofs to Lean proofs. The blueprint also had an accompanying diagram with all dependencies indicated, which made it easy to organise the project and see what was still to be done. To give an idea of the level of complexity of a formalisation of this length, the final dependendency graph for this project is shown in Figure~\ref{fig1}.

\begin{figure}
\includegraphics[scale=0.2]{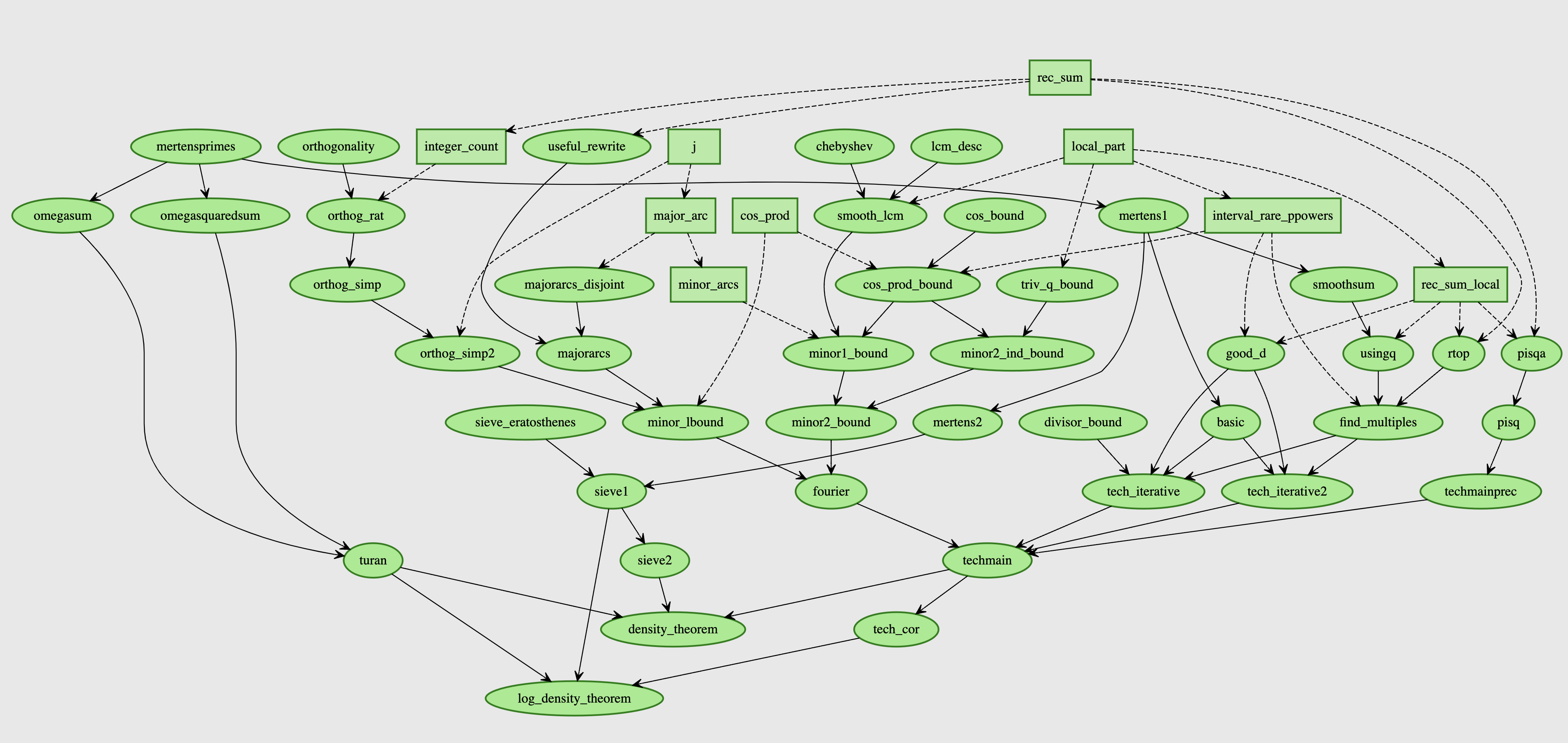}
\caption{The completed dependency graph for the formalisation.}\label{fig1}
\end{figure}

This blueprint is available at \url{https://b-mehta.github.io/unit-fractions/blueprint/index.html}. The template for this blueprint was created by Patrick Massot, and was also used for other Lean formalisation projects, such as the Liquid Tensor project.\footnote{\url{https://github.com/leanprover-community/lean-liquid}}

\subsection{Reflections} This formalisation is a first in several respects: it is the first recent analytic number theory result to be formally verified; the first instance of the circle method; the first solution to a long-standing problem of Erd\H{o}s.  Part of the motivation for this formalisation was as a proof of concept: the Lean proof assistant and accompanying \mathlib\ is advanced enough to make feasible the fast formalisation of new research results in mathematics, on the same timescale as the production of the `human-readable' paper. Of course, this was made feasible by the relatively elementary and self-contained nature of the mathematics involved. Nonetheless, we believe that this arrangement, with a formal certificate of validation accompanying the human version of the paper, is a sign of things to come.

\end{document}